%% file: YWJ11Dana.tex
\documentclass[11pt,oneside,fleqn]{article}

\usepackage{amsfonts}
\usepackage{amssymb}
\usepackage{fancyhdr}
\usepackage{comment}
\usepackage{natbib}
\usepackage{amsmath}
\usepackage{amsthm}
\usepackage{dsfont}
\usepackage{graphicx}
\usepackage{enumerate}
\usepackage{color, soul}
\usepackage{cases}
\newtheorem{theorem}{Theorem}

\newtheorem{lemma}{Lemma}

\usepackage{euscript}
\def\text#1{\hbox{\rm#1}}
\def\qt#1{\qquad\text{#1}}
\def\statespace{{\mathbb S}}
\def\statespaceT{\statespace_T}
\def\aa{{\EuScript A}}

\def\ee{{\EuScript E}}
\def\ff{{\EuScript F}}
\def\hh{{\EuScript H}}
\def\gmap{{\EuScript G}}
\def\nn{{\EuScript N}}

\def\ssYWJ{\statespace_{\text{\tiny\text{YWJ}}}}
\def\bone{\mathds{1}}

\def\PP{\mathbb{P}}
\def\RR{\mathbb{R}}
\def\UU{\mathbb{U}}	
\def\CORE{\mathbb{K}}	
\def\PPtheta{\mathbb{P}_\theta}
\def\That{\widehat{T}}
\def\SUM{\sum\nolimits}

\def\tlam{\widetilde\lambda}
\def\hops{\text{\scshape hops}}
\def\Sqrt{\mathpalette\DHLhksqrt}
\def\DHLhksqrt#1#2{%
\setbox0=\hbox{$#1\sqrt{#2\,}$}\dimen0=\ht0
\advance\dimen0-0.2\ht0
\setbox2=\hbox{\vrule height\ht0 depth -\dimen0}%
{\box0\lower0.4pt\box2}}

\newif\ifshowlabels	

\definecolor{refkey}{gray}{.75}
\definecolor{labelkey}{cmyk}{1,0,0,0.36}

\def\subheading#1{\vspace{10pt plus 1pt minus 1pt}\noindent{\bf #1}\newline}

\usepackage{hyperref}
	\hypersetup{
		colorlinks=true,
		linkcolor=blue,
		citecolor=blue,
		hypertexnames=false,
		pdfpagelayout=OneColumn, 
		pdfview=FitH, 
		pdfstartview=FitH, 
		implicit=false 
		}%

\begin{document}
%
%

\begin{center}

{\large \bf Rapid mixing of a Markov chain for an exponentially weighted aggregation estimator}

\bigskip

David Pollard  \; and \;  Dana Yang

\bigskip

\end{center}


\begin{abstract}
The Metropolis-Hastings method is often used to construct a Markov chain with a given~$\pi$  as its stationary distribution. The method works even if~$\pi$ is known only up to an intractable constant of proportionality.
Polynomial time convergence results for such chains (rapid mixing)  are hard to obtain for high dimensional probability models where the size of the state space potentially grows exponentially with the model dimension. 
In a Bayesian context, \citet*{yang2016computational} (=YWJ) used the path method to prove rapid mixing  for high dimensional linear models. 

This paper proposes a modification of the YWJ approach that 
simplifies the theoretical argument and improves the rate of convergence. The new approach is illustrated by an application to an exponentially weighted  aggregation estimation. 
\end{abstract}


\input{intro}
\input{mixing}
\input{chain}

\input{theorem}
\input{proofLemma}

\input{mixingupper}
\input{init}

\bibliographystyle{chicago}
\bibliography{dana}
\end{document}

%% file: intro.tex
\section{Introduction}\label{intro}
Many statistical problems involve sampling from a probability measure~$\pi$ defined  on a finite set~$\statespace$. For example, Bayesians are usually interested in the case where $\pi=\pi(\cdot \mid Y)$ is  a posterior distribution. In such settings~$\pi$ is often defined as a ratio of a simple numerator with a more complicated denominator, which can be computationally intractable if~$\statespace$ is large.
The Metropolis-Hastings (M-H) method provides one way to approach this problem.

With M-H one constructs a time reversible, irreducible, aperiodic Markov chain~$\{Z_n\}$ on~$\statespace$
with~$\pi$ as its stationary distribution. \
One starts with a ``proposal chain" given by a transition matrix $R$ then defines $P$ via acceptance/rejection of the proposal.  One defines 
\begin{equation}\label{M-H}
P(S,S') = R(S,S')\min \left\{1, \frac{\pi(S')R(S',S)}{\pi(S)R(S,S')}\right\}
\qt{for states $S\ne S'$}
\end{equation}
with the holding probability~$P(S,S)$ then chosen so that $\sum_{S'}P(S,S')=1$.

One hopes the  chain defined by~(\ref{M-H})
will converge rapidly to~$\pi$ even when~$\statespace$ is large.
It is known that the larger the `spectral gap' for~$P$ the faster the convergence. 
The  path method developed by~\cite{diaconis1991geometric} and~\cite{sinclair1992improved} provides lower bounds for the size of the spectral gap. These translate 
 easily into bounds on the mixing times, the number of steps of the chain $\{Z_n\}$ needed before the total variation distance between $\pi$ and the distribution of $Z_n$ becomes smaller than any specified $\epsilon>0$. See Section~\ref{mixing} for a more precise
statement of these results.



\citet*{yang2016computational} (= YWJ) used the M-H method to sample from a posterior distribution for a problem where the observed~$Y$ is modeled as having a $N(Xb,\sigma^2 I_n)$ distribution for an observed $n\times p$ matrix~$X$, with $p$ possibly much larger than $n$. They followed the tradition of studying behavior of the posterior under a model $\PPtheta$ for which $Y\sim N(X\theta, \sigma^2 I_n)$ for a sparse vector $\theta$, that is, a vector whose support set $T:=\{j\in [p]:\theta_j\neq 0\}$ they assumed to have cardinality no greater than some pre-specified (small) value $s^*\geq 1$. One of their main concerns was to determine how the rate of convergence of the Markov chain to its stationary distribution depended on $s^*,n,p,\theta$, the prior, and the various assumed properties of the matrix $X$. 
They used the path method to control the spectral gap.
Their Theorem~2 gave an $\epsilon$-mixing time of order~$O(s_0^2p\log p(n+s_0))$.


We had initially  hoped to adapt the YWJ approach to a 
similar-looking problem involving the aggregation estimators 
 described by \citet{rigollet2012sparse}. 
For this problem the observed~$Y$ is modeled as a sparse linear combination~$Xb$ plus a noise vector $\epsilon$. The estimator for the mean is taken as a convex combination $\sum_S \pi(S)\Phi_SY$ of least squares estimators, where~$\Phi_S$ denotes the matrix for orthogonal projection onto the subspace of $\mathbb{R}^n$ spanned by the columns of the $n\times|S|$ submatrix $X_S=(X_j:j\in S)$ 
of~$X$. Let $\|\cdot\|$ denote the Euclidean norm. The vector $\pi$ is defined to be of the form
\begin{equation}\label{pi}
\pi(S)\propto\mu(S)\exp\left(-\frac{\left\|(I-\Phi_S)Y\right\|^2+2\text{trace}(\Phi_S)}{\beta}\right),
\end{equation}
for some suitably chosen weight function~$\mu$ on~$\statespace$.
For $\beta=2$, the vector $\pi$ can be interpreted as a posterior distribution on the set of least squares projections $\{\Phi_SY\}_{S\subset [p]}$. 

Unfortunately we encountered some  technical difficulties in modifying the YWJ path construction. 
The YWJ chains ran on a state space whose elements they identified with subsets of columns of~$X$. More precisely, they took the prior distribution to concentrate on vectors~$b\in\mathbb{R}^p$ whose support $\{j\in [p]:b_j\neq 0\}$ belonged to a set
\begin{equation*}
\ssYWJ=\{S\in \{0,1\}^p:|S|\leq s_0\},
\end{equation*}
 for  some (suitably small)~$s_0$ depending on~$s^*$.

Here  we follow YWJ in identifying a subset of~$[p]$ with its indicator function, as an element of~$\{0,1\}^p$. The 
size~$|S|$ of a set~$S$ is equal to the number of ones in its indicator function. We also write~$h(S,S')$ for the Hamming distance between two sets, which coincides with the~$\ell^1$ distance between their indicator functions.

The restriction to small sets of columns was natural for YWJ, given the assumption of a sparse~$\theta$.
 However it had some unfortunate complicating effects on construction of the Markov chain and the paths that determine the mixing rate. The main difficulties arose for sets~$S$ on the ``boundary" of~$\ssYWJ$ as a subset of $\{0,1\}^p$, that is, the sets in~$\ssYWJ$ of size $s_0$. To keep the chain within~$\ssYWJ$ they had to invent a delicate 'single-flip$\backslash$double-flip' construction for their proposal chain. As far as we can tell, these double-flips would lead to a major slowdown for the analogous aggregation chain.

%

In this note we describe a modification of the YWJ approach that eliminates the difficulties caused by the boundary. We too analyse behavior under a fixed~$\PPtheta$, for~$\theta$ a vector with a sparse support set~$T$. In addition, we assume existence
of an estimator~$\widehat{T}$ that has high 
$\PPtheta$-probability of getting close to~$T$, in an appropriate sense.
As shown in Section~\ref{init}, the thresholded lasso estimator of~\cite{zhou2010thresholded} provides a suitable~$\That$.
Our M-H chain has state space~$\statespace=\{0,1\}^p$.
We use~$\That$ as the starting state. Instead of the hard boundary
for~$\ssYWJ$ we use a `soft boundary': when our proposal chain gets out to non-sparse regions of~$\statespace$ we allow jumps back to~$\That$ with probability~$1/2$. This choice prevents the chain from spending too much time exploring unpromising parts of the state space. We no longer need the double-flips for our proposal chain.
As shown by our Theorem~\ref{thm:main} in
Section~\ref{sec:theorem}, these choices lead to $\epsilon$-mixing times of order~$s^{*2}p\log p$, which is  faster than the rate achievable with the hard boundary.

A reanalysis of the YWJ problem using our soft-thresholding method  also improves on their mixing times. We omit such analysis
from this note and  refer interested readers to the thesis of
\citet{DanaPhD}, which contains a more detailed comparison of the 
methods.

%
%

%% file: mixing.tex
\section{Mixing times and the path method}\label{mixing}
Suppose $\pi$ is a stationary distribution for the transition matrix~$P$ of a Markov chain on a finite statespace~$\statespace$. Suppose also that the chain is time reversible:
for each pair of states~$S$ and~$S'$ in~$\statespace$,
\begin{equation*}
Q\{S,S'\}:=
\pi(S)P(S,S') = \pi(S')P(S',S).
\end{equation*}
Equivalently, $P$ corresponds to a random walk on a graph with 
vertices~$\statespace$  and edge weights $Q(\mathfrak{e})$ for $\mathfrak{e}=\{S,S' \}$ such that 
\begin{equation*}
\pi(S)= \SUM_{S'\in\statespace}Q\{S,S'\}\qquad\text{and}\qquad P(S,S')=Q\{S,S'\}/\pi(S).
\end{equation*}

For the M-H chains with~$P$ defined as 
in~(\ref{M-H}), the edge weight becomes
\begin{equation*}
Q\{S,S'\} = \min\left\{\pi(S)R(S,S'),\pi(S')R(S',S) \right\}.
\end{equation*}

Provided $P$ is irreducible and aperiodic, it has eigenvalues $1=\lambda_1> \lambda_2 \geq \dots \geq \lambda_N > -1$ where $N$ is the cardinality of the state space~$\statespace$. \citet[Proposition~3]{diaconis1991geometric} proved that, for such a $P$-chain started in state $S_0$, the $k$-step transition probabilities satisfy
\begin{equation*}
2\left\|P^k(S_0,\cdot) - \pi(\cdot)\right\|_{\rm TV} 
:= \sum_{S\in\statespace}|P^k(S_0,S)-\pi(S)|
\leq \pi(S_0)^{-1/2}\beta^k
\end{equation*}
where $\beta := \min(|\lambda_2|,|\lambda_N|)$.


The analysis is easier if one runs the `lazy' version of the chain,
with transition matrix $\widetilde{P}= (I+P)/2$, which has 
eigenvalues~$\tlam_i=(1+\lambda_i)/2$ for which $1 =\tlam_1 >\tlam_2 \ge \dots \ge \tlam_N\ge0$.
 The corresponding $\beta$ equals $(1+\lambda_2)/2\leq e^{-(1-\lambda_2)/2}$,
so that
\begin{equation*}
\left\|\widetilde{P}^k(S_0,\cdot) - \pi(\cdot)\right\|_{\rm TV} \leq \tfrac{1}{2}\pi(S_0)^{-1/2} e^{-k(1-\lambda_2)/2}.
\end{equation*}
The quantity $1-\lambda_2$ is called the spectral gap for the matrix $P$, which we denote by $\textsc{gap}(P)$. It is traditional to invert the last bound to see that $\left\|\widetilde{P}^k(S_0,\cdot) - \pi(\cdot)\right\|_{\rm TV}\leq \epsilon$ when $k\geq \tau_\epsilon(S_0)$ with
\begin{equation}\label{invert}
\tau_\epsilon(S_0) \leq \frac{2\log(1/2\epsilon)+\log(1/\pi(S_0))}{\textsc{gap}(P)}.
\end{equation}

For both the YWJ problem and the aggregation problem the challenge is to design chains for which $\textsc{gap}(P)$, does not decrease too rapidly to zero. 

The path method provides a lower bound for $\textsc{gap}(P)$. The method requires construction of a set of directed paths connecting different states, one path for each pair $(I,F)$ with $I\neq F$. The path $\gamma(I,F)$ connecting~$I$ and~$F$ should consist of distinct elements $S_0=I, S_1,\dots,S_m=F$ of the state space with edge weights $Q\{S_j,S_{j+1}\}>0$ for each $j$. The path can also be thought of as a sequence of directed edges, $(S_j,S_{j+1})$ for~$j=0,\dots,m$. The path has length $\textsc{len}(\gamma(I,F))=m$. The loading of a directed edge~$\mathfrak{e}$ is defined as
\begin{equation}\label{loading}
\rho(\mathfrak{e}) = \sum_{\gamma(I,F)\ni\mathfrak{e}}\pi(I)\pi(F)/Q(\mathfrak{e}),
\end{equation}
where the sum runs over all paths $\gamma$ with $\mathfrak{e}$ as one of their directed edges.

When the path $\gamma(I,F)$ is just the reverse of the path~$\gamma(F,I)$, as it was for YWJ and will be for us, the distinction between directed and undirected edges becomes less important.

\citet[Corollary 6]{sinclair1992improved} showed that
\begin{equation}\label{Sinclair}
1/\textsc{gap}(P) \leq \left( \max_{S,S'}\textsc{len}((\gamma(S,S'))\right) \times
\left( \max_\mathfrak{e} \rho(\mathfrak{e})\right).
\end{equation}

It is important to note that the paths are a theoretical construct that can depend on information about a Markov chain not  known to the MCMC practitioner. For example, the paths defined by YWJ were allowed to depend on the unknown mean $X\theta$ for the $N(X\theta,\sigma^2 I_n)$ distribution that generated~$Y$. Indeed they designed paths that involved knowledge of the support set $T=\{j: \theta_j\neq 0\}$.

%

%% file: chain.tex
\section{Our Metropolis-Hastings chain}\label{sec:chain}
We construct the transition matrix~$P$ as in in~(\ref{M-H}), for a chain with state space~$\statespace=\{0,1\}^p$,  where~$\pi$  
is constructed as  in~(\ref{pi}) from a weight function~$\mu$ inspired
by
the work of~\citet*{castillo2015bayesian} and~\citet*{gao2015general} on posterior contraction in the setting of high dimensional linear regression. For 
positive constants~$D$ and~$\beta$ (that need to be specified),
we define
\begin{equation*}
\mu(S)=\exp\left(-D|S|\log p-\frac{4n}{\beta}\bone\{|S|>4s^*\}\right)
\end{equation*}
so that
\begin{equation*}
\pi(S) = \exp\left( G(S,Y)-m(S)\right)/\mathcal{Z}(Y)
\end{equation*}
where
\begin{align*}
 G(S,y)&= \left\|\Phi_S Y\right\|^2/\beta
 \\
m(S) &= D |S|\log p + 2\text{trace}(\Phi_S)/\beta
+ (4n/\beta)\bone\{|S|>4s^*\}.
\end{align*}
The normalizing constant~$\mathcal{Z}(Y)$, which ensures that~$\sum_{S\in\statespace}\pi(S)=1$, is a complicated sum.
The indicator function ensures that states~$S$ with size greater than~$4s^*$ have very small~$\pi$ measure.

Write~$\nn(S)$ for the set of all~$p$ states~$S'$ at Hamming distance~$1$ from~$S$. We refer to a move from~$S$ to a state chosen uniformly at random from~$\nn(S)$ as a single flip.
Our proposal chain, with transition matrix~$R$, is constructed from a mixture
of single flips  and big jumps to a state $\widehat{T}$, which we will eventually assume has size at most $2s^*$
(with high $\PPtheta$ probability). Define~$\CORE=\{S\in\statespace: |S|\le 3s^*\}$ and
$\statespace_k = \{S\in\statespace: |S|=k\}$.
 Our proposal chain allows these moves:
\begin{enumerate}[(R1)]
\item 
If $S\in\CORE\backslash\{\widehat{T}\}$ then move via a single flip.
\item
If $S\in\CORE^c$ then, with probability $1/2$ move to $\widehat{T}$ and with probability $1/2$ move via a single flip.
 
\item
For a move from $\That$, with probability $1/2$ move via a single flip  and with probability $1/2$ first choose uniformly at random an integer $k$ with $3s^*< k \leq p$, then jump to an $S'$ chosen uniformly at random from~$\statespace_k$.
\end{enumerate}

It will be important to have $\pi(\That)$ not too small---the choice $\That=\emptyset$ does not work for our approach. The steps of the chain involving $\That$ are much easier to handle than the double flips of the YWJ method.

%% file: theorem.tex

\section{A lower bound on the spectral gap}\label{sec:theorem}
With notation as in Section~\ref{sec:chain}, define events:
\begin{align*}
\mathcal{A}_n&=\left\{|\widehat{T}|\leq 2s^*,\left\|(I-\Phi_{\widehat{T}})X\theta\right\|^2\leq cs^*\log p\right\}
\\
\mathcal{E}_n &=\left\{\max_{|S|<6s^*,j\notin S}\left|\left\langle(I-\Phi_S)X_j,\epsilon\right\rangle\right|^2\leq nL\nu\log p\right\}
\\
\mathcal{F}_n &= \{\|\epsilon\|^2\leq 2n\}
\\
\hh_n &=\aa_n\cap\ee_n\cap\ff_n .
\end{align*}
Here $c$ and $L$ are constants that are to be chosen and~$\nu$ is a constant that will appear in the statement of Theorem~\ref{thm:main}.

Our main theorem says something interesting only if the event~$\hh_n$ has high~$\PPtheta$ probability, which is true under reasonable assumptions:
\begin{enumerate}[(E1)]
\item 
See Section~\ref{init} for assumptions that ensure $\mathcal{A}_n$ occurs with high probability. 

\item
If~$\epsilon\sim N(0,I_n)$  then~$\|\epsilon\|^2\sim\chi_n^2$.
A very neat argument of \citet*[page~29]{BLM2013Concentration}
gives
$
\PP\{\|\epsilon\| \ge \sqrt n + \sqrt t \}\le e^{-t}$
for~$t\ge0$.
Thus $\PP\ff_n^c\le e^{-0.17 n}$.

\item
YWJ assumed that 
\begin{equation*}
\mathbb{E}_\theta\max_{|S|\leq s_0,j\notin S}\left|\left\langle(I-\Phi_S)X_j,\epsilon\right\rangle\right|\leq \sqrt{nL\nu\log p}/2, 
\end{equation*}
which ensures that  $\PPtheta\ee_n^c\le\exp(-L\nu\log p/8)$. 
\end{enumerate}

Our readers might prefer to add these as explicit assumptions to the next theorem, in which case the desired properties would be asserted to hold except on a set with impressively small probability.

\begin{theorem}\label{thm:main}
Assume
\begin{enumerate}[(i)]

\item\label{assump:design}
Each column of the design matrix~$X$ has ($\ell^2$) 
length~$\sqrt{n}$ ands there exists a constant $\nu>0$ such that
\begin{equation*}
\left\|X_S w\right\|^2\geq n\nu \|w\|^2\qquad\text{for each } w\in\mathbb{R}^S\text{ and }S\in \statespace\text{ with }|S|\leq 6s^*.
\end{equation*}

\item\label{assump:betamin}
$\min_{j\in T}|\theta_j|^2\geq \theta_{\min}^2
\geq \left(8\beta D \log p\right)/(n\nu^2)$.

\end{enumerate}
If 
\begin{equation}\label{D.choice}
D \ge 4 + (4L+2c)/\beta
\end{equation}
then  we have
$
1/\textsc{gap}(P)\leq 60ps^*$ on the set $\hh_n$.
\end{theorem}

YWJ assumption B required $\max_{|S|\leq s_0}\lambda_{\min}(X_S^TX_S/n)\geq\nu$, which is similar to our assumption~(\ref{assump:design}) except that they had the much larger $s_0$ in place of our $6s^*$ (see YWJ assumption D). 
YWJ
needed the larger $s_0$ value to accommodate their double flips. We are able to weaken their assumption by avoiding the difficulty around the boundary of the state space.

The proof of the Theorem will use the path method described in Section~\ref{mixing}.
We assume throughout the following argument that the sample puts us in~$\hh_n$. In particular, we assume~$|\That|\le 2s^*$.
We follow the idea of YWJ in constructing paths by means of a a map $\gmap:\statespace\backslash \{T\}\to \statespace$, but with a slightly different choice for $\gmap$. Our choice avoids the difficulties with the hard boundary.

Here is our construction for~$\gmap$.
Define 
\begin{align*}
\UU &:=\{S\in\statespace\backslash \{T\}: |S\backslash T|\le 3s^*\} 
\text{\quad and\quad}
\statespace_T := \{S\in\statespace\backslash \{T\}: S\supseteq T\}.
\end{align*}
Notice that
$
\CORE \subset \UU \subset \{S\in\statespace: |S|\le 4s^*\}$. 

\begin{figure}[h]
\label{fig:Gpaths}
\includegraphics[width = 3.5 in]{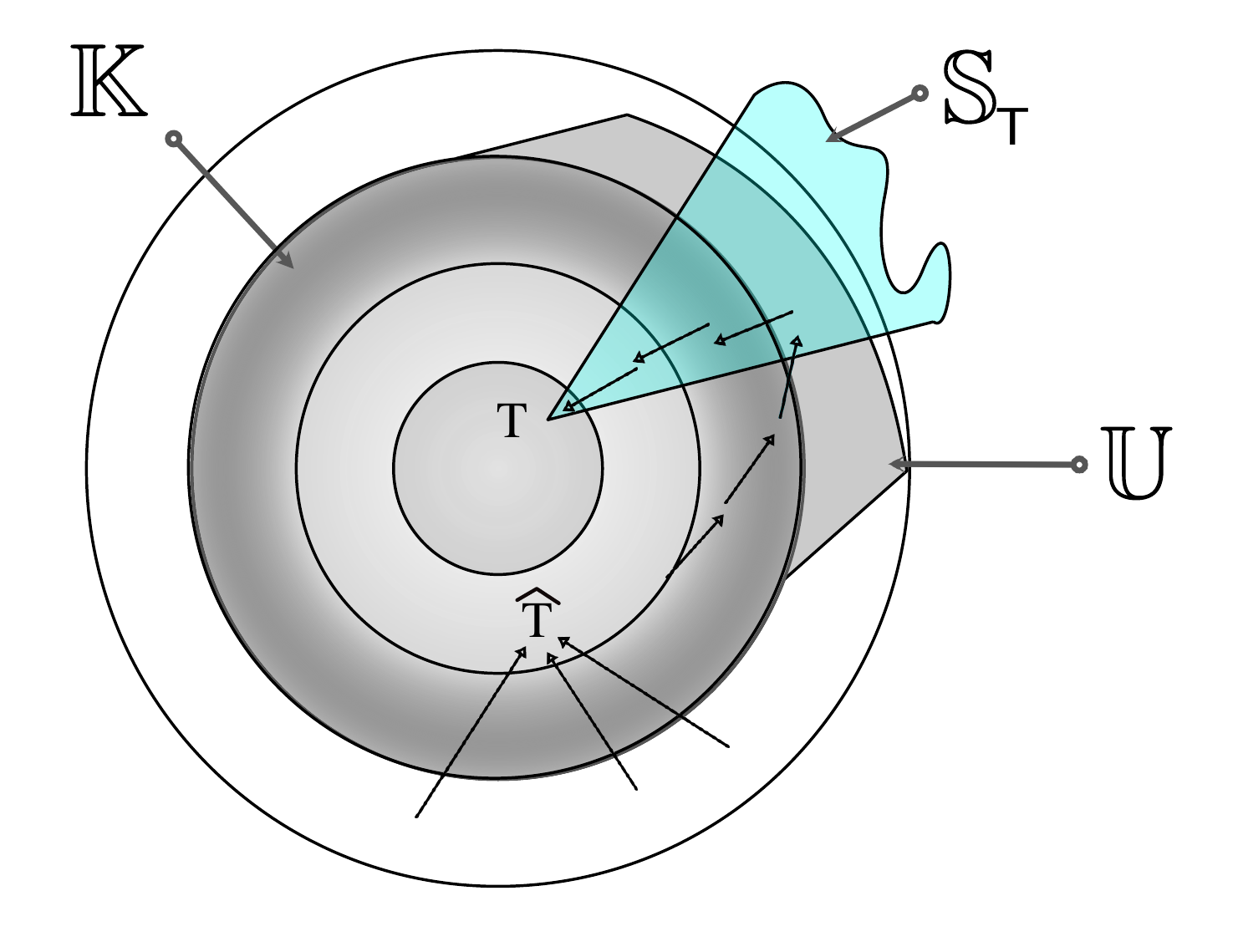}
\caption{A visualization of the subsets of the state space and the map $\gmap$. The empty set lies on the center of the rings and the size of a set is represented by its distance to the center. The four rings stand for sets of size $s^*$, $2s^*$, $3s^*$ and $4s^*$ from inside to outside. The arrows show the direction of the map $\gmap$.}
\end{figure}


\begin{enumerate}[($\gmap$1)]
\item 
If $S\in\statespace_T\cap\UU$ define $\gmap(S)=S\backslash\{i\}$ for an arbitrary~$i$ in~$S\backslash T$. (The choice of the particular~$i$ is not important but, for definiteness,  we could take it as the smallest~$j$ in~$S\backslash T$.)

\item
If $S\in\UU\backslash\statespace_T$
define~$\gmap(S)=S\cup\{i\}$ for the~$i$ in $T\backslash S$  
that gives the largest value for~$\left\|\Phi_{\gmap(S)} X\theta\right\|$.

\item
If $S\in\UU^c$ define $\mathcal{G}(S)=\widehat{T}$.

\end{enumerate}

The map~$\gmap$ defines a directed graph on~$\statespace$, with edges of the form~$(S,\gmap(S))$.
The choice of~$\UU$ ensures that $\gmap(S)\in\UU$ if~$S\in\UU\backslash\{T\}$.
Consequently, the $\gmap$-path from~$S$ to~$T$ stays inside~$\UU$ if~$S\in\UU\backslash\{T\}$.
It also ensures that~$\gmap$ decreases the Hamming distance to the true support,
\begin{equation*}
h(\gmap(S), T)<h(S,T)
\qt{for all $S\in\UU\backslash\{T\}$,
}
\end{equation*}
a property that implies the graph has no cycles. From each state, finitely many compositions of $\gmap$ eventually leads to $T$. That is, as for the YWJ construction, our map $\gmap$  defines a directed tree on $\statespace$ with $T$ as the root. See Figure~\ref{fig:Gpaths} for a visualization of the relationship between $\UU, \statespace_T$, $\CORE$ and the map $\gmap$.

For distinct states~$I$ and~$F$ we define~$\gamma(I,F)$ as the shortest path from~$I$ to~$F$ along the tree  (ignoring the $\gmap$-direction of the edges).
We write $d_\gmap(I,F)$ for the length of the~$\gamma(I,J)$ path, 
which is a metric on~$\statespace$.
Also following YWJ, we 
write
$\Lambda(S)$ for the set of all $S'$, including~$S$ itself, for which the $\gmap$-path from $S'$ to $T$ passes through~$S$.

Most of the hard work for the proof of Theorem~\ref{thm:main} occurs in deriving a bound for the loadings. The necessary facts are stated in the next Lemma, whose proof appears in 
Section~\ref{proofLemma}.

\begin{lemma}\label{pi.gmap}
On the event~$\hh_n$, under the assumptions of  Theorem~\ref{thm:main} the following inequalities hold for every~$S$ in~$\statespace\backslash\{T\}$.
\begin{enumerate}[(i)]
\item\label{PSggS}
$P(S,\gmap(S))\ge 1/(2p)$
\item\label{piLam}
$\pi(\Lambda(S))\le 3\pi(S)$
\end{enumerate}
\end{lemma}

The rest is relatively easy.

\begin{proof}[Proof of Theorem~\ref{thm:main}]
We assert that
\begin{align}
 &\max\nolimits_\mathfrak{e} \rho(\mathfrak{e}) \le 6p
\label{max.load}
\\
&\max\nolimits_{S,S'}\textsc{len}((\gamma(S,S'))\le 2+8s^*\leq 10s^*
\label{max.len}
\end{align}
from which it follows that
$
1/\textsc{gap}(P)\leq 60ps^*
$.

\subheading{The maximum load: proof of~\eqref{max.load}}
Suppose $I$ and~$F$ are distinct states. Let~$M$ denote the first state at which the~$\gmap$-path from~$I$ to~$T$ meets
the~$\gmap$-path from~$F$ to~$T$. The path~$\gamma(I,F)$ 
then consists of two segments: 
the~$\gmap$-path from~$I$ to~$M$ followed by the reverse of 
the~$\gmap$-path from~$F$ to~$M$.

Suppose $\mathfrak{e}=(S,\gmap(S))$ is a directed edge that appears
in~$\gamma(I,F)$. The state~$S$ must then lie on
the~$\gmap$-path from~$I$ to~$M$. It follows that
$I\in\Lambda(S)$ and $F\notin \Lambda(S)$. 
For each path~$\gamma(I,F)$ that contributes to the loading~$\rho(\mathfrak{e})$ the pair~$(I,F)$ must belong to $\Lambda(S)\times\Lambda(S)^c$. Thus
\begin{equation*}\label{eq:rho.rewrite}
\rho(\mathfrak{e})\le  \frac{
\sum_{I\in\Lambda(S),F\in\Lambda(S)^c}
\pi(I)\pi(F)}{Q(\mathfrak{e})}
=\frac{\pi(\Lambda(S))(1-\pi(\Lambda(S)))}{\pi(S)P(S,S')},
\end{equation*}
which is less than $6p$ by Lemma~\ref{pi.gmap}.

If $\mathfrak{e}=(\gmap(S),S)$ a similar argument, with  the roles of~$I$ and~$F$ interchanged, leads to the same upper bound.
 
\subheading{The maximum length: proof of~\eqref{max.len}}
First we bound the $d_\gmap$-distance from a generic~$S$ to~$T$. We consider a number of cases.

\begin{enumerate}[(a)]

\item \label{BST}
Claim: if $S\in\UU\cap\statespaceT$ then $d_\gmap(S,T) 
\le |S|-s^*\le 3s^*$.
Reason: It takes at most~\hbox{$|S|-s^*$} steps of type~$\gmap1$ to reduce~$S$ to~$T$.

\item \label{ThatT}
Claim: $d_\gmap(\That,T) \le 3s^*$. 
Reason: We know $|\That|\le 2s^*$. It takes at most~$s^*$ steps of type~$\gmap2$ to expand~$\That$ to~$\That\cup T$; and $d_\gmap(\That\cup T,T)\le 2s^*$ by case~(\ref{BST}).

\item \label{UUc}
Claim: if $S\in\UU^c$ then $d_\gmap(S,T)\le 1+3s^*$.
Reason: We have $d_\gmap(S,\That)=1$ because $\gmap(S)=\That$. 
The rest of the path from~$\That$ to~$T$ takes at most~$3s^*$ steps.

\item\label{UUnotST}
Claim: if~$S\in\UU\backslash \statespaceT$
then $d_\gmap(S,T)\le 4s^*$.
Reason: We build up $S$ to~$S\cup T\in\UU\cap\statespaceT$ by  at most~$s^*$ steps of type~$\gmap2$, then reduce to~$T$ as in case~(\ref{BST}).
\end{enumerate}
It follows, for each pair of distinct states~$I$ and~$F$, that
\begin{equation*}
d_\gmap(I,F) \le d_\gmap(I,T)+ d_\gmap(T,F) \le 2(1+4s^*).
\end{equation*}

%
%
%
%
%
%
\end{proof}

%% file: proofLemma.tex
\section{Proofs of technical lemmas}\label{proofLemma}
Throughout the Section we assume implicitly that the event~$\hh_n$ has occurred.

Several of the arguments rely on the following  simple consequence of assumption~(\ref{assump:design}) from Theorem~\ref{thm:main}. 
To simplify notation we write $A+B$ instead of~$A\cup B$ if ~$A$ and~$B$ are disjoint sets. Also, if~$j\in A^c$ and~$i\in A$ then we 
abbreviate~$A\cup\{j\}$ to~$A+j$ and $A\backslash\{i\}$ to~$A-i$.

\begin{lemma}\label{project}
 If $A$ and $B$ are disjoint subsets of $[p]$ with $|A|+|B|\leq 6s^*$ then 
\begin{enumerate}[(i)]
 \item\label{sing}
 $\left\|(I-\Phi_A)X_Bt\right\|\geq \sqrt{n\nu}\|t\|$ for each~$t$ in~$\mathbb{R}^B$
 
\item\label{eigen}
$\left\|X_B^T(I-\Phi_A)X_Bt\right\|\geq n\nu\|t\|$ for each~$t$ in~$\mathbb{R}^B$.

\item\label{one-step}
If $k\in A^c$  then $\Psi_{k; A}:=\Phi_{A+k}-\Phi_A=\|z_k\|^{-2}z_kz_k^T$ projects
vectors orthogonally onto the part of~$\text{span}(X_{A+k})$ that
is orthogonal to~$\text{span}(X_{A})$, the
one-dimensional subspace spanned by~$z_k :=(I-\Phi_A)X_k$.
For each~$w$ in~$\RR^n$,
\begin{equation*}
\left\|\Phi_{A+k}w\right\|^2-\left\|\Phi_A w\right\|^2
=\left\|\Psi_{k;A} w\right\|^2
\le \left\langle(I_n-\Phi_A)X_k, w\right\rangle^2/(n\nu).
\end{equation*}

\item\label{A+B.A}
If $B=\{k[1],\dots,k[m]\}$ and $A_j = A+\{k[\ell]: 1\le \ell\le j\}$ for $j=0,\dots,m$ then
\begin{equation*}
\left\|\Phi_{A+B}w\right\|^2-\left\|\Phi_A w\right\|^2
\le \sum_{0\le \ell < m}
\left\langle(I_n-\Phi_{A_{\ell}})X_{k[\ell+1]}, w\right\rangle^2/(n\nu).
\end{equation*}
\end{enumerate}
\end{lemma}
\begin{proof}
Assumption~(\ref{assump:design}) of Theorem~\ref{thm:main} and the bound $|A+ B|\le 6s^*$ give 
\begin{equation*}
\left\|X_Ar+X_Bt\right\|^2\geq n\nu(\|r\|^2+\|t\|^2)
\qt{for all $r\in\mathbb{R}^A$ and $t\in\mathbb{R}^B$.}
\end{equation*}
If we choose $r$ so that $X_Ar=-\Phi_AX_Bt$ then ignore the $\|r\|^2$ on the right-hand side we are left with
\begin{equation*}
\left\|(I-\Phi_A)X_Bt\right\|\geq n\nu\|t\|^2\qquad\text{for each }t\text{ in }\mathbb{R}^B,
\end{equation*}
which is equivalent to~(\ref{sing}).
The same inequality also
implies that  the $n\times |B|$ matrix $W=(I-\Phi_A)X_B$ has smallest singular value no less than~$\sqrt{n\nu}$. It follows that
the positive definite matrix~$W^TW=X_B^T(I-\Phi_A)X_B$ has smallest eigenvalue no less than~$n\nu$, which is equivalent to~(\ref{eigen}).

The first equality in~(\ref{one-step}) comes from the decomposition of~$\Phi_{A+k}w$ into the sum of orthogonal components~$\Psi_{k;A}w+\Phi_Aw$. The inequality comes from the lower bound~$\|z_k\|\ge \Sqrt{n\nu}$, which is a special case of~(\ref{sing}) with~$B$ replaced by~$\{k\}$ and~$t=1$.

For inequality~(\ref{A+B.A}) define 
\begin{equation*}
z_{\ell+1}=\left(I-\Phi_{A_{\ell}} \right)X_{k[\ell+1]}
\qt{for $0\le \ell<m$.}
\end{equation*}
Note that $\Phi_{A+B}-\Phi_A$ is a sum of projections onto the 
one-dimensional orthogonal spaces spanned by the~$z_{\ell+1}$'s.
Then invoke~(\ref{one-step}).
\end{proof}

%
%
%
%

The proof of Lemma~\ref{pi.gmap}
requires control of the ratio~$\pi(S)/\pi(S')$ for various pairs~$S,S'$. The necessary facts are contained in the following lemma.  It is here that the main technical differences between the YWJ argument and ours appear.


\begin{lemma}\label{ratio}
Under the assumptions of Theorem~\ref{thm:main}, 
for all~$S\in\statespace$:
\begin{equation}\label{S.That}
\log\frac{\pi (S)}{\pi (\widehat{T})}
 -\tfrac12D|S|\log p + Cs^*\log p
 \qt{where $C= 2D+4/\beta +2c/\beta$}.
\end{equation}
and for each $S$ in $\statespace\backslash\{T\}$:
\begin{numcases}{\log\frac{\pi(S)}{\pi(\mathcal{G} (S))} \leq}
 -\tfrac12 D\log p & if $S\in \UU$
\label{eq:S.ggS.U}
\\
-2|S|\log p& if $S\in\UU^c$.\label{eq:S.ggS.Ucomp}
\end{numcases}
Consequently
\begin{equation}\label{gmap.UU}
\pi(\gmap(S)) \ge \exp\left(\tfrac12 D\log p \right)\pi(S)\ge \pi(S)
\qt{for all $S$ in $\UU$.}
\end{equation}

\end{lemma}

\begin{proof}[Proof of~\eqref{eq:S.ggS.U}]
First consider the case where $S\in \UU\cap \statespaceT$. By construction $\mathcal{G}(S)=S-i$ for some~$i$ in $S\backslash T$. Temporarily write~$A$ for $\gmap(S)$, so that $S=A+i$. Note that $A\supseteq T$, which implies 
\begin{equation*}
(I-\Phi_A)X\theta=(I-\Phi_A)X_T\theta_T=0.
\end{equation*}
 Thus
\begin{align*}
\beta&\left( G(S,Y)-G(\gmap (S), Y)\right)
\\ 
&= \|\Phi_{A+i}Y\|^2-\|\Phi_{A}Y\|^2
\\
&\le \left\langle(I_n-\Phi_A)X_i, Y\right\rangle^2/(n\nu)
\qt{by Lemma~\ref{project}(\ref{one-step})}
\\
&= \left\langle(I_n-\Phi_A)X_i, \epsilon\right\rangle^2/(n\nu)
\qt{because $A\supseteq T$}
\\
&\le L\log p
\qt{bound from event $\ee_n$.}
\end{align*}

The contribution from the dimension penalty is even easier to handle:
\begin{align*}
m(A+j)-m(A)&= D\left( |A+j|-|A|\right)\log p  +2\left(\text{trace}(\Phi_{A+i})-\text{trace}(\Phi_A) \right)/\beta
\\
&\ge D\log p .
\end{align*}
Thus
\begin{align*}
\log \frac{\pi(S)}{\pi(\gmap(S))}
&=  G(A+i,Y)-G(A,Y)- m(A+i)+m(A)
\\
&\leq (L/\beta)\log p -D\log p
\\
&\leq 
-\tfrac12{D\log p}
\qt{if }D>2L/\beta.
\end{align*}

The proof for the $S\in \UU\backslash\statespaceT$ case borrows heavily from YWJ Section~B.4]. 
This time the set $B:=T\backslash S$ is nonempty and, by construction, $\gmap(S)=S\cup\{i\}$ where $i=\arg\max_{j\in B}\left\|\Phi_{S+j}X\theta\right\|$. 

Temporarily define~$A=S$. For each~$j$ in~$B$ Lemma~\ref{project}(\ref{one-step}) gives
\begin{equation*}
\Psi_j :=\Phi_{A+j}-\Phi_A = \|z_j\|^{-2}z_jz_j^T
\qt{where }z_j =(I-\Phi_A)X_j.
\end{equation*}
Note that each~$\Psi_j$ kills any component of~$X\theta$ contributed by the columns of~$X_A$. Thus $\Psi_j X\theta=\Psi_j X_B\theta_B$ for each~$j$ in~$B$.
Note also that~$i$  maximizes
\begin{align*}
\|\Phi_{A+j}X\theta\|^2 -\|\Psi_A X\theta\|^2 
&= \|\Psi_j X_B\theta_B\|^2.
\end{align*}

To bound $\pi(S)/\pi(\gmap (S))$ from above we need to bound~$\pi(A+i)/\pi(A)$ from below. We start by deriving a lower bound for
\begin{equation}\label{betaG}
\beta\left( G(A+i,Y)-G(A,Y)\right)
=\left\|\Psi_i\left(X_B\theta_B +\epsilon\right)\right\|^2.
\end{equation}
By the triangle inequality,
\begin{equation}\label{triangle}
\left\|\Psi_i\left(X_B\theta_B +\epsilon\right)\right\|
\ge \left\|\Psi_i X_B\theta_B\right\| 
- \left\|\Psi_i\epsilon\right\|
\end{equation}
As in the proof of the $S\in \UU\cap \statespaceT$ case, we have
$
\left\|\Psi_i\epsilon\right\|^2  \le L\log p$ on the event~$\ee_n$.

For the contribution from~$X_B\theta_B$ use the fact that a maximum is greater than an average:
\begin{align*}
\|\Psi_i X_B\theta_B\|^2
&\ge \frac{1}{|B|}\sum_{j\in B}\|\Psi_j X_B\theta_B\|^2
\\
&= \frac{1}{|B|}\sum_{j\in B}
\dfrac{(X_B\theta_B)^Tz_jz_j^TX_B\theta_B}{\|z_j\|^2}
\\
&\ge \dfrac{\sum_{j\in B}(X_B\theta_B)^Tz_jz_j^TX_B\theta_B}{|B|n}
\qt{because }\|z_j\|^2 \le \|X_j\|^2= n.
\end{align*}
The vectors $\{z_j: j\in B\}$ are the columns of the matrix~$Z_B=(I-\Phi_A)X_B$. The sum~$\sum_{j\in B}z_jz_j^T$ equals~$Z_BZ_B^T$. It follows that
\begin{align*}
\|\Psi_i X_B\theta_B\|^2 
&\ge \frac{\|Z_B^TX_B\theta_B\|^2}{|B|n}
\\
&= \frac{\|X_B^T(I-\Phi_A)X_B\theta_B\|^2}{|B|n}
\\
&\ge (n\nu)^2\|\theta_B\|^2/(|B|n)
\qt{by Lemma~\ref{project}(ii)}
\\
&\ge n\nu^2 \min_{j\in B}|\theta_j|^2
\\
&\ge 8\beta D \log p
\qt{by assumption (ii) of Theorem~\ref{thm:main}.}
\end{align*}
Inequalities~(\ref{triangle})  and the choice of~$D$ to satisfy~(\ref{D.choice})
now imply
\begin{equation*}
\left\|\Psi_i\left(X_B\theta_B +\epsilon\right)\right\|
\ge \Sqrt{8\beta D \log p} -\Sqrt{L\log p}
\ge \tfrac34\Sqrt{8\beta D \log p}.
\end{equation*}
so that
\begin{equation*}
G(A+i,Y)-G(A,Y)
\ge \tfrac92 D\log p.
\end{equation*}

As in the proof of the $S\in \UU\cap\statespaceT$ case we also have
$
m(A+i)-m(A)=  D\log p +2/\beta 
$.
It follows that
\begin{equation*}
\log\frac{\pi(A+i)}{\pi(A)} \ge \tfrac72 D\log p - 2/\beta \geq \tfrac12 D\log p.
\end{equation*}

\end{proof}

\begin{proof}[Proof of~\eqref{S.That}] 
The set $\widehat{T}$ may not be contained in the set $S$. We use a little trick to remedy the problem. Write~$A$ for~$\That$ and~$B$ for~$S\backslash \That$. Then
\begin{align*}
\beta\left(G(S,Y)-G(\widehat{T},Y) \right)&= \|\Phi_S Y\|^2-\|\Phi_{\widehat{T}}Y\|^2
\\
&\leq  \left\|\Phi_{S\cup \widehat{T}} Y\right\|^2-\|\Phi_{\widehat{T}}Y\|^2
\\
 &= \left\|\left(\Phi_{A+B}-\Phi_A\right)Y\right\|^2\\
&\leq  \left(\left\|\left(\Phi_{A+B}-\Phi_A\right)X\theta\right\|+\left\|\left(\Phi_{A+B}-\Phi_A\right)\epsilon\right\|\right)^2
\\
&\leq  2\left\|\left(I-\Phi_A\right)X\theta\right\|^2
+2\left\|\left(\Phi_{A+B}-\Phi_A\right)\epsilon\right\|^2.
\end{align*}

On the event $\mathcal{A}_n$, we have $\left\|\left(I-\Phi_A\right)X\theta\right\|^2\leq cs^*\log p$. 
For the $\epsilon$ contribution we consider two cases.
If $|S|> 4s^*$ then
\begin{equation*}
\left\|\left(\Phi_{A+B}-\Phi_A\right)\epsilon\right\|^2 \le \|\epsilon\|^2 \leq 2n.
\end{equation*}
If~$|S|\le 4s^*$ then~$|S\cup\That|\le 6s^*$. Define sets~$A_\ell$
as in part~(\ref{one-step}) of Lemma~\ref{project}. Then
\begin{align*}
\left\|\left(\left(\Phi_{A+B}-\Phi_A\right)\right)\epsilon\right\|^2
&\le \sum_{0\le \ell < m}
\left\langle(I_n-\Phi_{A_{\ell}})X_{k[\ell+1]}, \epsilon\right\rangle^2/(n\nu),
\end{align*}
which is less than $mL\log p$ where $m=|S\cup \widehat{T}|-|\widehat{T}|\le |S|$.
In summary,
\begin{equation*}
\beta\left(G(S,Y)-G(\widehat{T},Y) \right)
\le 2cs^*\log p + 4n\mathds{1}\left\{|S|>4s^*\right\}
+ \left(2|S| L\log p \right) \mathds{1}\left\{|S|\le4s^*\right\}.
\end{equation*}

%
There is no problem with set inclusion
for the dimension penalization terms:
\begin{align*}
& \beta\left(m(S)-m(\That) \right)\\
= & \beta D(|S|-|\widehat{T}|)\log p+2(\text{trace}(\Phi_S)-\text{trace}(\Phi_{\widehat{T}}))+4n\mathds{1}\left\{|S|>4s^*\right\}\\
\geq & \beta D(|S|-2s^*)\log p-4s^*+4n\mathds{1}\left\{|S|>4s^*\right\}.
\end{align*}
Subtraction then yields
\begin{equation}
\label{S.That.stronger}
\log\tfrac{\pi(S)}{\pi(\widehat{T})}\leq -D(|S|-2s^*)\log p+\frac{4s^*+2cs^*\log p+2|S|L\log p}{\beta},
\end{equation}
from which~\eqref{S.That} follows.
%
%
%
\end{proof}

\begin{proof}[Proof of~\eqref{eq:S.ggS.Ucomp}]
For $S\in \UU ^C$, the size of $S$ is larger than $3s^*$. From~\eqref{S.That.stronger}, 
\begin{align*}
\log\frac{\pi(S)}{\pi(\That)}
&\leq -D(|S|-2s^*)\log p+
\beta^{-1}\left(4s^* +2cs^*\log p + 2|S| L\log p \right)
\\
&\leq -\left(D-2L/\beta \right)|S|\log p
+ \left(2D+4/\beta +2c/\beta \right)s^*\log p
\\
&\leq -\tfrac13\left(D-(6L+4+2c)/\beta \right)|S|\log p,
\end{align*}
which is less than $-|S|\log p$ by the choice for~$D$.
\end{proof}

\begin{proof}[Proof of~\eqref{gmap.UU}]
$
-D\log p + 2/\beta \le -\tfrac12 D \log p\le 0
$.
\end{proof}

\begin{proof}[Proof of Lemma~\ref{pi.gmap}(\ref{PSggS}): $P(S,\gmap(S))\ge 1/(2p)$ for every~$S$ in~$\statespace\backslash\{T\}$]
 \ \newline
 We consider two cases, both for the event~$\hh_n$.
 
If $S\in\UU\backslash \{T\}$ then~$\gmap(S)\in\UU\cap\nn(S)$,
so that $R(S,\gmap(S))= R(\gmap(S),S)\geq 1/(2p)$ and
\begin{align*}
P(S,\gmap(S))&=R(S,\gmap(S))
\min \left\{1, 
\frac{\pi(\gmap(S))R(\gmap(S),S)}{\pi(S)R(S,\gmap(S))}
\right\}
\\
&\geq  (2p)^{-1}\min \left\{1, \frac{\pi(\gmap(S))}{\pi(S)}\right\}
\\
&= (2p)^{-1}
\qt{by inequality \eqref{gmap.UU}. }
\end{align*}


For $S\in\UU^c $ we have $\mathcal{G}(S)=\widehat{T}$ and
\begin{align*}
R(S,\That)&= 1/2
\\
R(\That, S) &= \frac{1}{2}\left((p-2s^*)\binom{p}{|S|}\right)^{-1}
\ge (1/p)^{|S|+1}
\\
\pi(S)/\pi(\That) &\le \exp\left( -2|S|\log p\right)
\end{align*}
Thus
\begin{align*}
P(S, \widehat{T})
&=R(S,\That)\min\left\{1, \frac{\pi(\widehat{T})R(\widehat{T},S)}{\pi(S)R(S,\widehat{T})}\right\}
\\
&\ge\tfrac12\min\left\{1, \exp\left( 2|S|\log p - (1+|S|)\log p)\right)\right\} =\tfrac12.
\end{align*}
%
\end{proof}

\begin{proof}[Proof of Lemma~\ref{pi.gmap}(\ref{piLam}): 
$\pi(\Lambda(S))\le 3\pi(S)$ for every~$S$ 
in~$\statespace\backslash\{T\}$]
\ \newline 
Again we consider two cases, both for the event~$\hh_n$.

If $S\in\UU^c$ then $\Lambda(S)=\{S\}$ and the asserted inequality
holds trivially.

If~$S\in\UU$
we split the set $\Lambda(S)$ in two parts: $\Lambda(S)\cap \UU $ and $\Lambda(S)\backslash \UU $. 
If~$\Lambda(S)\backslash \UU \ne \emptyset$ then there exists 
an~$S'\in\UU^c$ for which~$S$ lies on the~$\gmap$-path from~$S'$ to~$T$. In particular, $\That=\gmap (S')$ must belong to~$\Lambda(S)$.
By virtue of inequality~\eqref{gmap.UU} we then have~$\pi(\That)\le\pi(S)$ so that
\begin{align*}
\frac{\pi(\Lambda(S)\backslash \UU)}{\pi(S)} 
&\le \frac{\pi(\UU^c)}{\pi(\That)} 
= \sum_{k\ge 3s^*}\frac{\pi(S'\in\UU^c:|S'|=k)}{\pi(\That)} 
\\
&\le \sum_{k\ge 3s^*}\binom{p}{k}\exp(-2k\log p)
\qt{by inequality~\eqref{eq:S.ggS.Ucomp}}
\\
&\le 1.
\end{align*}
 
Each~$S'$ in $\Lambda(S)\cap \UU$ is connected to~$S$ by a $\gmap$-path that stays inside~$\UU$. We write $\hops(S')$ for the number of edges in that path. 
By inequality~\eqref{gmap.UU} the value of~$\pi$ increases by at least a factor of~$\exp\left( \tfrac12 D \log p\right)$ for each edge.
Across each edge the size of the set is changed by~$1$, by either the deletion~($\gmap$1) 
or the addition~($\gmap$2) of one vertex. If $\hops(S')=k$ then~$S'$ must lie within Hamming distance~$k$ of~$S$. There are at most
\begin{equation*}
\binom p0 +\binom p1+\dots+\binom p k \le \left( ep/k\right)^k
\end{equation*}
such~$S'$ sets.
It follows that
\begin{align*}
\frac{\pi(\Lambda(S)\cap \UU )}{\pi(S)}
&= \frac{\pi(S)}{\pi(S)} + \sum_{k\ge1} 
\frac{\pi(S')}{\pi(S)}\mathds{1}\left\{S'\in\Lambda(S)\cap \UU: \hops(S')=k\right\}
\\
&\le 1+ \sum_{k\ge1} \left( ep/k\right)^k\exp\left( -\tfrac12 kD\log p\right)
\\
&\le 2
\qt{by the choice of $D$.}
\end{align*}
%
%
%
%
%
%
%
%
%
\end{proof}

%% file: mixingupper.tex

\section{An upper bound on the mixing time}\label{sec:mixing.time}
From equation~\eqref{invert} we have
\begin{equation*} 
\tau_\epsilon(\That) = \frac{2\log(1/2\epsilon)+\log(1/\pi(\That))}{\textsc{gap}(P)}.
\end{equation*}
Theorem~\ref{thm:main} showed that
$1/\textsc{gap}(P)\leq 60ps^*$ on the set $\hh_n$.  To get a bound on~$\tau_\epsilon(\That)$ we only need to show that $\pi(\widehat{T})$ is not too small.

\begin{theorem}\label{cor:mix.time}
Under the assumptions of Theorem~\ref{thm:main}, 
on the set $\hh_n$,
\begin{equation*}
\tau_\epsilon(\widehat{T})\leq 120ps^*\left(\log \frac{1}{2\epsilon}+2Ds^*\log p\right).
\end{equation*}
\end{theorem}
\begin{proof}

As shown at the start of the proof of
 inequality~\eqref{eq:S.ggS.Ucomp}, inequality~\eqref{S.That}
can be rewritten as
\begin{equation*} 
\log\frac{\pi (S)}{\pi (\widehat{T})}
\le 
 -\tfrac12D|S|\log p + Cs^*\log p
 \qt{where $C= 2D+4/\beta +2c/\beta$}.
\end{equation*}
There are are $\displaystyle\binom pk$ sets~$S$ of size~$k$. Thus
\begin{align*}
\frac{1}{\pi(\That)} &= \sum_{S\in\statespace}\frac{\pi(S)}{\pi(\That)} 
\le \exp\left( Cs^*\log p\right)
\sum_{k\ge 0}\binom pk \exp\left( -\tfrac12Dk\log p\right)
\end{align*}
The choice of~$D$ ensures that the sum converges.
%
%
%
\end{proof}

%% file: init.tex

\section{Choice of initializer}\label{init}

Theorem~\ref{thm:main} holds for all initializers $\widehat{T}$ in the set
\begin{equation*}
\mathcal{A}_n=\left\{|\widehat{T}|\leq 2s^*,\left\|(I-\Phi_{\widehat{T}})X\theta\right\|^2\leq cs^*\log p\right\}.
\end{equation*}

When $s_1$ is taken to be a constant multiple of $s^*$,~\cite{zhou2010thresholded} showed that the thresholded LASSO estimator falls in $\mathcal{A}_n$ with high probability under mild assumptions. For completeness we will give the form of the estimator and the proof for controlling its prediction risk here.

Write $\lambda_n$ for  $\sqrt{\log p/n}$. The LASSO estimator is defined as
\begin{equation*}
\widehat{\theta}=\arg\min_{t}\|Y-Xt\|^2+\alpha\lambda_n \|t\|_1.
\end{equation*}

Define $\delta=\widehat{\theta}-\theta$. \citet*[Theorem~7.2]{bickel2009simultaneous} showed that under the restricted eigenvalue condition $\kappa=\kappa(s^*,3)>0$, on a set with probability at least $1-p^{1-\alpha^2/32}$,
\begin{equation}\label{eq:lasso.bound}
\|\delta\|_1\leq \frac{8\alpha\lambda_n}{\kappa^2}s^*,\;\;\;\|\delta_T\|\leq \frac{2\alpha\lambda_n}{\kappa^2}\sqrt{s^*}.
\end{equation}

We define $\widehat{T}$ to be $\{j:|\widehat{\theta}_j|>8\alpha\lambda_n/\kappa^2\}$. The following theorem restates a result of~\citet*[Theorem~1.3]{zhou2010thresholded}. It provides a theoretical guarantee that the event $\mathcal{A}_n$ occurs with high probability for this choice of $\widehat{T}$.

\begin{theorem}
Under the REC$(s^*,3)$ condition, on a set with probability at least $1-p^{1-\alpha^2/32}$, we have $|\widehat{T}|\leq 2s^*$ and
\begin{equation*}
\left\|(I-\Phi_{\widehat{T}})X\theta\right\|\leq \frac{2\alpha\sqrt{\Lambda_{\max}(s^*)}}{\kappa^2}\sqrt{s^*\log p},
\end{equation*}
where $\Lambda_{\max}=\max_{|S|\leq s^*}\lambda_{\max}(X_S^TX_S/n)$.
\end{theorem}
\begin{proof}

To handle the size of $\widehat{T}$, note that
\begin{equation*}
|\widehat{T}|=|\widehat{T}\cap T|+|\widehat{T}\backslash T|.
\end{equation*}
The first inequality in~\eqref{eq:lasso.bound} implies that
\begin{equation*}
\frac{8\alpha\lambda_n}{\kappa^2}s^*\geq \|\delta\|_1\geq \sum_{j\in \widehat{T}\backslash T}|\delta_j|>|\widehat{T}\backslash T|\cdot \frac{8\alpha\lambda_n}{\kappa^2},
\end{equation*}
the last inequality comes from the fact that $|\delta_j|=|\widehat{\theta}_j|>8\alpha\lambda_n/\kappa^2$ for all $j\in \widehat{T}\backslash T$. It follows that $|\widehat{T}\backslash T|\leq s^*$, and therefore $\widehat{T}\leq |T|+s^*\leq 2s^*$.

For the prediction risk:
\begin{equation*}
\left\|(I-\Phi_{\widehat{T}})X\theta\right\|= \left\|(I-\Phi_{\widehat{T}})X_{T\backslash \widehat{T}}\theta_{T\backslash \widehat{T}}\right\|\leq \left\|X_{T\backslash \widehat{T}}\theta_{T\backslash \widehat{T}}\right\|\leq \sqrt{n\Lambda_{\max}}\|\theta_{T\backslash \widehat{T}}\|.
\end{equation*}
From the second inequality in~\eqref{eq:lasso.bound}, we have $\|\theta_{T\backslash \widehat{T}}\|\leq \|\delta_T\|\leq 2\alpha\lambda_n\sqrt{s^*}/\kappa^2$. Conclude that
\begin{equation*}
\left\|(I-\Phi_{\widehat{T}})X\theta\right\|\leq \sqrt{n\Lambda_{\max}}\frac{2\alpha\lambda_n}{\kappa^2}\sqrt{s^*}=\frac{2\alpha\sqrt{\Lambda_{\max}(s^*)}}{\kappa^2}\sqrt{s^*\log p}.
\end{equation*}

\end{proof}